\documentclass[11pt]{amsart}
\usepackage{geometry}                
\usepackage{amsmath}
\usepackage{amscd}
\usepackage{amsthm}
\usepackage{graphicx}
\usepackage{amssymb}
\usepackage{epstopdf}
\theoremstyle{plain}
\newtheorem{Thm}{Theorem}[section]

\newtheorem{Prop}[Thm]{Proposition}
\newtheorem{Cor}[Thm]{Corollary}
\newtheorem{Lem}[Thm]{Lemma}

\theoremstyle{definition}
\newtheorem{Defn}[Thm]{Definition}
\newtheorem{Expl}[Thm]{Example}

\theoremstyle{Remark}
\newtheorem{Rem}[Thm]{Remark}
\newtheorem{Que}[Thm]{Question}

\numberwithin{equation}{section}

\title{Non-commutative deformations of simple objects \\
in a category of perverse coherent sheaves}
\author{Yujiro Kawamata}
\begin{document}
\maketitle

\tableofcontents

\begin{abstract}
We determine versal non-commutative deformations of some simple collections in the categories
of perverse coherent sheaves arising from tilting generators for projective morphisms.
13D09, 14B10, 14E30, 16E35. 
\end{abstract}

\section{Introduction}

Let $k$ be an algebraically closed field of characteristic $0$, and
let $f: Y \to X$ be a projective morphism of Noetherian $k$-schemes such that $X = \text{Spec}(R)$ is an affine scheme.
A locally free coherent sheaf $P$ on $Y$ is called a {\em tilting generator} if the following conditions are satisfied: 
(1) all higher direct images of $\mathcal{E}nd(P)$ for $f$ vanishes, (2) $P$ generates the bounded 
derived category of coherent sheaves $D^b(\text{coh}(Y))$ (see Definition~\ref{tilting}). 
Then the derived Morita equivalence theorem of Bondal and Rickard (\cite{Bondal}, \cite{Rickard}) 
tells us that there is an equivalence of triangulated categories
$D^b(\text{coh}(Y)) \cong D^b(\text{mod-}A)$, where $A = f_*\mathcal{E}nd(P)$ is a coherent sheaf of associative 
$\mathcal{O}_X$-algebras.
Let $\text{Perv}(Y/X)$ be the abelian subcategory of $D^b(\text{coh}(Y))$ corresponding to the category of
coherent right $A$-modules $(\text{mod-}A)$. 

We would like to call $\text{Perv}(Y/X)$ the category of {\em perverse coherent sheaves}.
By definition, the category of coherent sheaves $(\text{coh}(Y))$ is more geometric and 
$\text{Perv}(Y/X)$ more algebraic.
For example, for a $k$-valued point $x_0 \in X$, the set of simple objects in $(\text{coh}(Y))$ 
above a simple object $\mathcal{O}_{x_0}$ in $(\text{coh}(X))$ is identified as 
the set-theoretic fiber $f^{-1}(x_0)$ and is an infinite set in general.
On the other hand, such a set in $\text{Perv}(Y/X)$ is finite since $A$ is coherent.

The original perverse sheaves (\cite{BBD}) are complexes in a derived category of constructible sheaves
which correspond to sheaves of regular holonomic $D$-modules by the Riemann-Hilbert correspondence
(\cite{Kashiwara}, \cite{Mebkout}).
It was defined by generalizing the intersection homology theory of Goresky and MacPherson (\cite{GM})
which was discovered in the pursuit of homology theory for singular spaces which behaves better under 
the Poincare duality. 
Our definition is similar in that the perverse coherent sheaves are complexes of coherent sheaves which
correspond to sheaves over associative algebras by the Bondal-Rickard equivalence.
The perverse coherent sheaves possess more algebraic nature due to the construction, and 
we can expect that they behave better under certain problems.

In this paper, we consider the multi-pointed non-commutative deformations of sets of simple objects 
in $\text{Perv}(Y/X)$ (\cite{NC}),
and determine the versal deformations in some cases.

The versal deformation of the simple objects over a closed point of the base space recovers the associative algebra:

\begin{Thm}[=Theorem~\ref{recover}]
Let $\{s_j\}_{j=1}^m$ be the set of all simple objects in $\text{Perv}(Y/X)$ above a closed point $x_0 \in X$.
Let $\hat P$ be the completion of the direct sum of all 
indecomposable projective objects $\{P_i\}_{i=1}^m$ in the category of perverse coherent sheaves 
$\text{Perv}(\hat Y/\hat X)$
on the completion $\hat Y = Y \times_X \hat X$ above $x_0$.
Then $\hat P$ is the versal deformation of the simple collection $\bigoplus_{j=1}^m s_j$ with the
parameter algebra $\hat A = f_*\mathcal{E}nd(\hat P)$.
\end{Thm}

The versal deformation of a partial collection recovers the {\em contraction algebra} of 
Donovan-Wemyss (\cite{Donovan-Wemyss1}) in the case of 
perverse coherent sheaves of Bridgeland and Van den Bergh (\cite{Bridgeland}, \cite{VdBergh}):

\begin{Thm}[=Theorem~\ref{defBvdB}]
Assume that the dimension of the fibers of $f$ are at most $1$, and that
$Rf_*\mathcal{O}_Y = \mathcal{O}_X$.
Let $C$ be the scheme theoretic closed fiber of $f$ above $x_0$, 
and let $C_i$ ($i = 1,\dots, r$) be the irreducible components of $C$.

(A) Let $\{P_i\}_{i=0}^r$ and $\{s_j\}_{j=0}^r$ be the sets of indecomposable projective objects and simple objects
in ${}^{-1}\text{Perv}(\hat Y/\hat X)$ defined in \S 3 (A).
Let $\hat P = \bigoplus_{i=0}^r P_i$, and let $Q$ be the kernel of the natural homomorphism 
$p: f^*f_*\hat P \to \hat P$.
Let $I$ is the two-sided ideal of $\hat A = f_*\mathcal{E}nd(\hat P)$ generated by endomorphisms of $\hat P$
which can be factored in the form $\hat P \to P_0 \to \hat P$. 
Then $Q[1] \in {}^{-1}\text{Perv}(\hat Y/\hat X)$ 
is the versal deformation of the simple collection $\bigoplus_{j=1}^r s_j$
with the parameter algebra $\text{End}(Q[1]) \cong \hat A/I$.

(B) Let $\{P'_i\}_{i=0}^r$ and $\{s'_j\}_{j=0}^r$ be the sets of indecomposable projective objects and simple objects
in ${}^0\text{Perv}(\hat Y/\hat X)$ defined in \S 3 (B).
Let $\hat P' = \bigoplus_{i=0}^r P'_i$, and let $Q'$ be the cokernel of the 
natural homomorphism $p': f^*f_*\hat P' \to \hat P'$.
Let $I'$ is the two-sided ideal of $\hat A' = f_*\text{End}(\hat P')$ generated by endomorphisms of $\hat P'$
which can be factored in the form $\hat P' \to P'_0 \to \hat P'$. 
Then $Q' \in {}^0\text{Perv}(\hat Y/\hat X)$
is the versal deformation of the simple collection $\bigoplus_{j=1}^r s'_j$
with the parameter algebra 
$\text{End}(Q') \cong \hat A'/I'$.
\end{Thm}

Indeed we have $s_j = s'_j[1]$ and $Q=Q'$, though $\hat P \ne \hat P'$  
and ${}^{-1}\text{Perv}(\hat Y/\hat X) \ne {}^0\text{Perv}(\hat Y/\hat X)$.

The {\em contraction algebra} $\hat A/I \cong \hat A'/I'$ is an important invariant of $f$.
For example, Donovan and Wemyss conjectured that, in the case of a flopping contraction of a smooth $3$-fold $Y$, 
the contraction algebra determines the singularity of $X$ (see \cite{Hua-Toda}, \cite{Hua}, \cite{Hua-Zhou} 
for the development).

As a corollary, we prove that $\hat A/I$ is isomorphic to its opposite ring (Corollary~\ref{opposite}).
We also have the following:

\begin{Cor}[= Corollary~\ref{isolated}]
Assume in addition that $f$ is a birational morphism.
Then $f$ is an isomorphism outside the fiber over $x_0$ in a neighborhood of the fiber if and only if 
the parameter algebra $\hat A/I$ of the versal deformation of the simple collection $\bigoplus_{j=1}^r s_j$
is finite dimensional as a $k$-vector space.  
\end{Cor}

The above result was proved by Donovan and Wemyss (\cite{Donovan-Wemyss1}, \cite{Donovan-Wemyss2})
in the case where $f$ is a 
flopping contraction of a $3$-fold, i.e., $X$ and $Y$ are $3$-dimensional algebraic varieties having only 
Gorenstein terminal singularities and $f$ is a crepant morphism 
in the sense that $(K_Y,C_i) = 0$ for all $i$.
Bodzenta and Bondal \cite{BB} extended their result to the case where 
$X$ is a hypersurface of multiplicity $2$ with canonical singularities and $Y$ is Gorenstein.

We assume that all schemes and morphisms in this paper are defined over a fixed algebraically 
closed base field $k$ of characteristic $0$.

The author would like to thank Kohei Yahiro for letting him know the argument in Lemma~\ref{yahiro}.
This work was partly done while the author stayed at National Taiwan University.
The author would like to thank Professor Jungkai Chen and 
National Center for Theoretical Sciences of Taiwan of the hospitality and 
excellent working condition.

This work is partly supported by Grant-in-Aid for Scientific Research (A) 16H02141.

\section{Perverse coherent sheaves}

\begin{Defn}\label{tilting}
Let $f: Y \to X$ be a projective morphism between noetherian schemes such that $X = \text{Spec}(R)$ 
is an affine scheme.
A locally free coherent sheaf $P$ on $Y$ is said to be a {\em tilting generator} for $f$ 
if the following conditions are satisfied:

\begin{enumerate}
\item $R^pf_*\mathcal{E}nd(P) = 0$ for $p > 0$.

\item For $a \in D^b(\text{coh}(Y))$, if $\text{Hom}(P,a[p]) \cong 0$ for all $p$, then $a \cong 0$. 
\end{enumerate}
\end{Defn}

\begin{Thm}[Bondal-Rickard (\cite{Bondal}, \cite{Rickard})]\label{BR}
Let $f: Y \to X = \text{Spec}(R)$ be a projective morphism between noetherian schemes,  
let $P$ be a tilting generator for $f$, and 
let $A = f_*\mathcal{E}nd(P)$ be the endomorphism algebra.
Then $A$ is an associative $\mathcal{O}_X$-algebra which is coherent as an $\mathcal{O}_X$-module, and 
there is an equivalence of triangulated categories
\[
\Phi: D^b(\text{coh}(Y)) \to D^b(\text{mod-}A)
\]
given by an exact functor $\Phi(\bullet) = R\text{Hom}(P, \bullet)$.
The quasi-inverse functor $\Psi: D^b(\text{mod-}A) \to D^b(\text{coh}(Y))$ is given by
$\Psi(\bullet) = \bullet \otimes^L_A P$.
\end{Thm}

The theorem connects (algebraic) geometry and (non-commutative) algebra.

We mainly consider the case where $X$ is a spectrum of a noetherian complete local ring.
This is justified by the following proposition:

\begin{Prop}
Let $f: Y \to X = \text{Spec}(R)$ be a projective morphism between noetherian schemes,  
let $x \in X$ be a closed point, let $\hat X$ be the completion of $X$ at $x$, let $\hat Y = Y \times_X \hat X$ 
be the fiber product, and let $\hat f: \hat Y \to \hat X$ be the induced morphism.

(1) Let $P$ be a tilting generator for $f$, and 
let $A = f_*\mathcal{E}nd(P)$ be the endomorphism algebra.
Let $\hat P$ be the locally free sheaf on $\hat Y$
induced from $P$, and let $\hat A = \hat f_*\mathcal{E}nd(\hat P)$ be the endomorphism algebra.
Then $\hat P$ is a tilting generator for $\hat f$, and $\hat A = \mathcal{O}_{\hat X} \otimes_{\mathcal{O}_X} A$.

(2) Let $\bigoplus_{j=1}^r s_j$ be a simple collection on $Y$ whose support is contained in a fiber $f^{-1}(x)$.
Then its non-commutative deformations on $Y$ are the same as those on $\hat Y$. 
\end{Prop}

We define the abelian category of perverse coherent sheaves:

\begin{Defn}
Let $f: Y \to X = \text{Spec}(R)$ be a projective morphism between noetherian schemes, and
let $P$ be a tilting generator for $f$.
The category of {\em perverse coherent sheaves} $\text{Perv}(Y/X)$ for $f$ is the abelian category 
corresponding to the category of finitely generated right $A$-modules $(\text{mod-}A)$
by $\Psi$.
In other words,
\[
\text{Perv}(Y/X) = \{a \in D^b(\text{coh}(Y)) \mid \text{Hom}(P,a[p]) \cong 0 \text{ for } p \ne 0\}.
\]
\end{Defn}

Thus the standard $t$-structure on $D^b(\text{mod-}A)$ is transferred to a $t$-structure on $D^b(\text{coh}(Y))$
whose heart is the category of perverse coherent sheaves. 

We would like to call an object in $\text{Perv}(Y/X)$ a {\em perverse coherent sheaf}.
It is not a sheaf nor perverse.
This is a generalization of such a category defined in the seminal paper by Bridgeland 
(\cite{Bridgeland}) and extended by Van den Bergh (\cite{VdBergh}).
We note that the category $\text{Perv}(Y/X)$ depends on the choice of the tilting generator $P$.
Indeed we obtain ${}^p\text{Perv}(Y/X)$ with different perversities $p \in \mathbf{Z}$ in their papers 
for different choices of $P$.
For example, if $L$ is an invertible sheaf on $Y$, then the auto-equivalence of $D^b(\text{coh}(Y))$ defined
by $\bullet \mapsto \bullet \otimes L$ sends the tilting generator $P$ and the category of perverse coherent sheaves
$\text{Perv}(Y/X)$ to different ones.  

By definition, we have the following:

\begin{Prop}
Let $f: Y \to X = \text{Spec}(R)$ be a projective morphism between noetherian schemes, 
let $P$ be a tilting generator for $f$, and let $\text{Perv}(Y/X)$ be the corresponding category of 
perverse coherent sheaves.
Then $P \in \text{Perv}(Y/X)$, and $P$ become its {\em projective generator}, i.e., 
$P$ is a projective object in this abelian category and $\text{Hom}(P,a) \cong 0$ for $a \in \text{Perv}(Y/X)$ 
implies that $a \cong 0$.
\end{Prop}

The category of perverse coherent sheaves $\text{Perv}(Y/X)$ 
has a different nature from the category of coherent sheaves $(\text{coh}(Y))$
in the sense that there are only finitely many points, or 
simple objects, above a point of a base space $X$. 

The author learned the following lemma from Dr. Kohei Yahiro:

\begin{Lem}\label{yahiro}
Let $X = \text{Spec }R$ be the spectrum of a noetherian complete local ring $R$ 
whose residue field is isomorphic to the base field, 
and let $A$ be an associative $R$-algebra which is finitely generated as an $R$-module.
Then the numbers of mutually non-isomorphic simple objects and mutually 
non-isomorphic indecomposable projective objects 
are finite in the abelian category $(\text{mod-}A)$.
Moreover their numbers, say $m$, are equal.
Let $s_1,\dots, s_m$ (resp. $P_1, \dots, P_m$) be such simple objects 
(resp. indecomposable projective objects).
Then $\dim \text{Hom}(P_i,s_j) = \delta_{ij}$ after possible permutations of indexes.
\end{Lem}

\begin{proof}
Let $J$ be the Jacobson radical of $A$, i.e., the intersection of all maximal right ideals.
Since $J$ is a fully invariant right submodule of $A$, it is invariant under automorphisms by left multiplications, 
hence a two-sided ideal.
Then $\bar A = A/J$ becomes an associative Artin algebra.
$\bar A$ is semi-simple, and is isomorphic as an $\bar A$-module to a direct sum of finitely many simple modules.
It follows that simple $\bar A$-modules and indecomposable projective $\bar A$-modules are the same, 
hence the assertion of the lemma for $\bar A$.

A simple $A$-module is the same as a simple $\bar A$-module with the natural $A$-module structure.
On the other hand, since $A$ is $J$-adically complete, the map  
from the set of finitely generated projective $A$-modules to the set of finitely generated projective 
$\bar A$-modules given by $P \mapsto P \otimes_A \bar A$ is bijective (\cite{Bass}~III-2.12).
Therefore the lemma is proved.     
\end{proof}

If $X = \text{Spec}(R)$ is the spectrum of a complete local ring, 
then the original tilting generator $P$ is a direct sum of the indecomposable projective objects $P_i$.
The reduced sum $\bar P = \bigoplus_{i=1}^m P_i$ is also a tilting generator which gives the same 
category of perverse coherent sheaves.
We call $\bar P$ the {\em reduced tilting generator}.

\section{The case of Bridgeland and Van den Bergh}

We recall the definition of perverse coherent sheaves by Bridgeland 
(\cite{Bridgeland}) and Van den Bergh (\cite{VdBergh}).

Let $f: Y \to X$ be a projective morphism between noetherian $k$-schemes.
Assume the following conditions:

\begin{enumerate}
\item $X = \text{Spec}(R)$ for a complete local ring $R$ whose residue field is isomorphic to $k$.

\item The dimension of the closed fiber of $f$ is equal to $1$.

\item $Rf_*\mathcal{O}_Y = \mathcal{O}_X$.

\end{enumerate}

Let $C$ be the scheme theoretic closed fiber of $f$, 
and let $C_i$ ($i = 1,\dots, r$) be the irreducible components of $C$.
By the assumption that $R^1f_*\mathcal{O}_Y = 0$, we have $C_i \cong \mathbf{P}^1$ for all $i$.
Let 
\[
\bar{\mathcal{C}} = \{c \in D^b(\text{coh}(Y)) \mid Rf_*c = 0\}
\]
and $\mathcal{C} = \bar{\mathcal{C}} \cap (\text{coh}(Y))$.
By the spectral sequence 
\[
E_2^{p,q} = R^pf_*H^q(c) \Rightarrow R^{p+q}f_*c
\]
we deduce that, for $c \in D^b(\text{coh}(Y))$, we have $c \in \bar{\mathcal{C}}$ if and only if 
$H^p(c) \in \mathcal{C}$ for all $p$.

\begin{Rem}
(1) $\mathcal{C}$ is an abelian category.
Let $h: c_0 \to c_1$ be a morphism in $\mathcal{C}$, i.e, a homomorphism of coherent sheaves such that 
$R^if_*c_j = 0$ for $i,j=0,1$. 
Let $c'_0 = \text{Ker}(h)$ and $c'_1 = \text{Coker}(h)$ in the category of coherent sheaves.
Then we claim that $c'_j \in \mathcal{C}$ for $j=0,1$.

Indeed $c'_0 \subset c_0$ implies $f_*c'_0 = 0$.
Since $f_*\text{Im}(h) = 0$ and $R^1f_*c_0 = 0$, we have $R^1f_*c'_0 = 0$.
$c'_1$ is treated similarly.

(2) But $\bar{\mathcal{C}} \not\cong D^b(\mathcal{C})$.
For example, assume that $f: Y \to X$ is a contraction of a smooth rational curve $C$ in a smooth $3$-fold 
whose normal bundle is 
isomorphic to $\mathcal{O}_{\mathbf{P}^1}(-1) \oplus \mathcal{O}_{\mathbf{P}^1}(-1)$.
Then $C$ is rigid, and $\mathcal{C}$ is equivalent to the category of $k$-vector spaces 
generated by $\mathcal{O}_C(-1)$.
On the other hand, $\text{Hom}^3(\mathcal{O}_C(-1), \mathcal{O}_C(-1)) \cong k$ in $D^b(\text{Coh}(Y))$.
\end{Rem}

We define the following categories of perverse coherent sheaves.

\vskip 1pc

(A) Let $P_0 = \mathcal{O}_Y$, 
and let $L_i$ for $i = 1,\dots, r$ be line bundles on $Y$ such that $(L_i,C_j) = \delta_{ij}$. 
We define locally free sheaves $\tilde P_i$ on $Y$ by exact sequences:
\begin{equation}\label{def Pi}
0 \to \mathcal{O}_Y^{\oplus r_i} \to \tilde P_i \to L_i \to 0
\end{equation}
such that the induces homomorphisms
\[
\text{Hom}(\mathcal{O}_Y^{\oplus r_i}, \mathcal{O}_Y) \to \text{Ext}^1(L_i,\mathcal{O}_Y) 
\]
are surjective. 
Then $\tilde P = \bigoplus_{i=0}^r \tilde P_i$ is a tilting generator, 
and the corresponding category of perverse coherent sheaves is denoted by ${}^{-1}\text{Perv}(Y/X)$.

The number $-1$ is the {\em perversity};
we have $\mathcal{C}[1] \subset {}^{-1}\text{Perv}(Y/X)$.
More precisely, we have
\begin{equation}\label{Perv-1}
\begin{split}
{}^{-1}\text{Perv}&(Y/X) \\ 
= \{&E \in D^b(\text{coh}(Y)) \mid 
f_*H^{-1}(E) = 0, R^1f_*H^0(E) = 0, \text{Hom}(H^0(E),\mathcal{C}) = 0 \\
&H^p(E) = 0 \text{ for } p \ne 0,-1\}. 
\end{split}
\end{equation}

Let $\tilde P_i \cong P_i \oplus P_0^{r'_i}$ be a decomposition into indecomposable sheaves.
Let $s_0 = \mathcal{O}_C$ and $s_j = \mathcal{O}_{C_j}(-1)[1]$ for $j > 0$.
Then $\{P_i\}_{i=0}^r$ and $\{s_j\}_{j=0}^r$ are the sets of indecomposable projective objects and 
simple objects in ${}^{-1}\text{Perv}(Y/X)$.

\vskip 1pc

(B) Let $P'_0 = \mathcal{O}_Y$, and define locally free sheaves $\tilde P'_i$ by exact sequences
\begin{equation}\label{def P'i}
0 \to \tilde P'_i \to \mathcal{O}_Y^{\oplus r_i} \to L_i \to 0
\end{equation} 
such that the induces homomorphisms
\[
\text{Hom}(\mathcal{O}_Y, \mathcal{O}_Y^{\oplus r_i}) \to \text{Hom}(\mathcal{O}_Y,L_i) 
\]
are surjective. 
Then $\tilde P' = \bigoplus_{i=0}^r \tilde P'_i$ is a tilting generator, 
and the corresponding category of perverse coherent sheaves is denoted by ${}^0\text{Perv}(Y/X)$.

The number $0$ is the perversity;
we have $\mathcal{C} \subset {}^0\text{Perv}(Y/X)$.
More precisely, we have
\begin{equation}\label{Perv0}
\begin{split}
{}^0\text{Perv}&(Y/X) \\ 
= \{&E \in D^b(\text{coh}(Y)) \mid 
f_*H^{-1}(E) = 0, \text{Hom}(\mathcal{C}, H^{-1}(E)) = 0, R^1f_*H^0(E) = 0 \\
&H^p(E) = 0 \text{ for } p \ne 0,-1\}. 
\end{split}
\end{equation}

Let $\tilde P'_i \cong P'_i \oplus (P'_0)^{r'_i}$ be a decomposition into indecomposable sheaves.
Let $s'_0 = \omega_C[1]$, the shift of the dualizing sheaf of $C$, and $s'_j = \mathcal{O}_{C_j}(-1)$ for $j > 0$.
Then $\{P'_i\}_{i=0}^r$ and $\{s'_j\}_{j=0}^r$ are the sets of indecomposable projective objects and 
simple objects in ${}^0\text{Perv}(Y/X)$.

\vskip 1pc

We have $P'_i \cong P_i^* = \text{Hom}(P_i,\mathcal{O}_Y)$.
Hence $A' \cong A^o$, the opposite algebra where the addition is the same but the multiplication is reversed.

\section{Other examples}

We consider divisorial contractions in this section.

\subsection{Contraction of a projective space}

Let $f: Y \to X = \text{Spec}(R)$ be a projective birational morphism from a smooth variety to a variety 
with an isolated singularity 
whose exceptional locus is a prime divisor $E \cong \mathbf{P}^{n-1}$
with normal bundle $N_{E/Y} \cong \mathcal{O}_E(-d)$ for some $d > 0$.
$X$ has a terminal singularity if $d < n$, and 
$f$ is crepant if $d = n$.
The line bundles $\mathcal{O}_E(i)$ can be extended to line bundles $P_i = \mathcal{O}_Y(i)$ for integers $i$.

\begin{Prop}\label{P-vanishing}
$P = \bigoplus_{i=0}^{n-1} P_i$ is a tilting generator of $D^b(\text{coh}(Y))$.
\end{Prop}

\begin{proof}
For any small positive number $\epsilon$, the pair $(Y,(1-\epsilon)E)$ is log terminal.
By \cite{KMM}~Theorem~1.2.5, we have 
$R^pf_*\mathcal{O}_Y(i) = 0$ for $p > 0$ and $i > -n$, 
because $\mathcal{O}_E(K_Y + E) \cong \mathcal{O}_E(-n)$.
Therefore $R^pf_*(P_i^* \otimes P_j) = 0$ for $p > 0$ and $0 \le i,j \le n-1$. 

By \cite{BB2}~Lemma~4.2.4 or \cite{VdBergh}~Lemma~3.2.2, 
$D^b(\text{coh}(Y))$ is generated by the $\mathcal{O}_Y(i)$ for $0 \le i \le n-1$.
\end{proof}

We denote by $\text{Perv}(Y/X)$ the corresponding category of perverse coherent sheaves.

\begin{Prop}
$s_j=\Omega^j_E(j)[j]$ for $0 \le j \le n-1$ are the simple objects of $\text{Perv}(Y/X)$
above the singular point of $X$ 
such that $\text{Hom}(P_i,s_j) \cong k^{\delta_{ij}}$.
\end{Prop}

\begin{proof}
By \cite{Beilinson}, there is a resolution of the diagonal $\Delta_E \subset E \times E$:
\[
0 \to \mathcal{O}_E(-n+1) \boxtimes \Omega^{n-1}_E(n-1) \to 
\dots \to \mathcal{O}_E(-1) \boxtimes \Omega^1_E(1) \to 
\mathcal{O}_E \boxtimes \mathcal{O}_E \to \mathcal{O}_{\Delta_E} \to 0.
\]
We define an integral functor $\Phi: D^b(\text{coh}(E)) \to D^b(\text{coh}(E))$ by 
$\Phi(\bullet) = p_{1*}(p_2^*\bullet \otimes \mathcal{O}_{\Delta_E})$.
Since $\Phi(\mathcal{O}_E(-i)) \cong \mathcal{O}_E(-i)$, we deduce that $\mathcal{O}_E(-i)$ is 
quasi-isomorphic to a complex
\[
\mathcal{O}_E(-n+1) \otimes R\Gamma(E,\Omega^{n-1}_E(n-1-i)) \to 
\dots \to \mathcal{O}_E(-1) \otimes R\Gamma(E,\Omega^1_E(1-i)) \to 
\mathcal{O}_E \otimes R\Gamma(E,\mathcal{O}_E(-i))
\]
where the last term is put at degree $0$.
Therefore, for $0 \le i,j \le n-1$, we have
\[
R\Gamma(E,\Omega^j_E(j-i)) \cong \begin{cases}
0 &\text{ if } i \ne j \\
k[-j] &\text{ if } i = j.
\end{cases}
\]
Thus the proposition is proved.
\end{proof}

We can obtain a similar result for the Grassmannian variety $G(r,n)$ if we use \cite{Kapranov}
instead of \cite{Beilinson}.

\vskip 1pc

$X$ has only one quotient singularity.
Let $\tilde X$ be the associated Deligne-Mumford stack.
Then there is a fully faithful functor $D^b(\text{coh}(Y)) \to D^b(\text{coh}(\tilde X))$ if
and only if $K_Y \le f^*K_X$ (\cite{DK}, \cite{SDG}).
This inequality is equivalent to saying that $X$ is not terminal, or $d \ge n$.

An associative algebra $A$ is said to be {\em homologically homogeneous} if the homological dimension 
of all simple objects are equal.
If $A$ is homologically homogeneous, then $A$ has finite global dimension and is Cohen-Macaulay 
(\cite{BH}).
Van den Bergh (\cite{VdBergh}) defined that $X = \text{Spec}(R)$ has a {\em non-commutative crepant resolution} 
if there is a reflexive $R$-module $M$ such that $A = \text{End}(M)$ is homologically homogeneous.

These facts correspond to the following:

\begin{Prop}
$A= f_*\mathcal{E}nd(P)$ is Cohen-Macaulay if and only if $d \ge n$.
\end{Prop}

\begin{proof}
$A$ is Cohen-Macaulay if and only if $A$ satisfies the condition that $\text{Ext}^p(A,\omega_X) = 0$ for $p > 0$, 
where $\omega_X$ is the canonical sheaf of $X$.
Since $R^pf_*\mathcal{E}nd(P) = 0$ for $p > 0$, we have
$Rf_*\mathcal{E}nd(P) \cong A$.
By the Grothendieck duality theorem, our condition is equivalent to that 
$R^pf_*\mathcal{E}nd(P)(K_Y) = 0$ for $p > 0$, 
where we have $\mathcal{O}_Y(K_Y) \cong \mathcal{O}_Y(d-n)$.
Since $\mathcal{E}nd(P)$ is a direct sum of the $\mathcal{O}_Y(m)$ for $-n+1 \le m \le n-1$,
it follows that $\mathcal{E}nd(P)(K_Y)$ is a direct sum of the $\mathcal{O}_Y(m)$ for $d -2n + 1 \le m \le d-1$.
We have 
$R^pf_*\mathcal{O}_Y(m) = 0$ for $p > 0$ and $m \ge -n+1$ by the proof of Proposition~\ref{P-vanishing}.
On the other hand, since $H^{n-1}(E,\mathcal{O}_E(-n)) \cong k$, we have 
$R^{n-1}f_*\mathcal{O}_Y(-n) \ne 0$.
Hence we obtain our assertion.
\end{proof}

\begin{Expl}
Assume that $n = 2$. 
By the above lemma, $A$ is homologically homogeneous if and only if $d \ge 2$; 
we calculate the homological dimension as follows:

\begin{enumerate}
\item $h.d.(s_0) = 1$ if $d = 1$ and $=2$ if $d \ge 2$. 

\item $h.d.(s_1) = 2$.
\end{enumerate}

Indeed we can check the assertions in the following way.

(1) Since $\mathcal{O}_E(K_Y) \cong \mathcal{O}_E(d-n)$, we have 
$\text{Hom}^2(\mathcal{O}_E,\mathcal{O}_E)^* \cong \text{Hom}(\mathcal{O}_E,\mathcal{O}_E(d-2)) \ne 0$ 
for $d \ge 2$, hence $h.d.(s_0) \ge 2$, where we denote $\text{Hom}^p(a,b) = \text{Hom}(a,b[p])$.
On the other hand, we note that, since $Y$ has global dimension $2$, we have $h.d.(s_0) \le 2$.
If $d=1$, then there is an exact sequence 
\[
0 \to \mathcal{O}_Y(1) \to \mathcal{O}_Y \to \mathcal{O}_E \to 0
\]
where $\mathcal{O}_Y$ and $\mathcal{O}_Y(1)$ are projective objects.
Hence $h.d.(s_0) = 1$.

(2) We have again $\text{Hom}^2(\mathcal{O}_E(-1)[1],\mathcal{O}_E(-1)[1])^* 
\cong \text{Hom}(\mathcal{O}_E,\mathcal{O}_E(d-2)) \ne 0$ for $d \ge 2$.
If $d = 1$, then $\text{Hom}^2(\mathcal{O}_E(-1)[1],\mathcal{O}_E) 
= \text{Hom}^1(\mathcal{O}_E(-1),\mathcal{O}_E) \ne 0$, 
since there is a non-trivial exact sequence
\[
0 \to \mathcal{O}_E(1) \to \mathcal{O}_{2E} \to \mathcal{O}_E \to 0.
\]
\end{Expl}

\subsection{Contraction of a singular quadric surface}

Let $X$ be a singularity of dimension $3$ defined by an equation
\[
xy+z^2+w^3=0
\]
in the completion of $\mathbf{C}^4$ at the origin, and let $f: Y \to X$ be the blowing up of the origin.
$Y$ is smooth, and the exceptional divisor $E$ of $f$ is a singular quadric surface with a singular point $Q$.

Let $\mathcal{O}_E(m)$ for $m \in \mathbf{Z}$ be a reflexive sheaf of rank $1$ on $E$ 
corresponding to a Weil divisor $mC_1$, where $C_1$ is a line on the cone $E$.
$mC_1$ is a Cartier divisor if and only if $m$ is even.
Let $\hat{\Omega}^1_E$ be the double dual of the sheaf of K\"ahler differentials on $E$.
It is a reflexive sheaf of rank $2$ on $E$ with a short exact sequence
\[
0 \to \hat{\Omega}^1_E \to \mathcal{O}_E(-1)^{\oplus 2} \oplus \mathcal{O}_E(-2) \to \mathcal{O}_E \to 0
\]
because $E \cong \mathbf{P}(1,1,2)$.
It follows that, for the double dual of the sheaf of 
differential $2$-forms, we have $\hat{\Omega}^2_E \cong \mathcal{O}_E(-4)$. 

Let $L$ be a line bundle of $Y$ such that $L \vert_E = \mathcal{O}_E(2)$, let $S \in \vert L \vert$ be a 
generic member, and let $C_2 = S \cap E \in \vert \mathcal{O}_E(2) \vert$.
We take a generic curve $C_0$ on $S$ such that $C_0 \cap C_2 = C_1 \cap C_2$ scheme theoretically.

It is known that there is a non-trivial extension 
\[
0 \to \mathcal{O}_E(1) \to \bar P_1 \to \mathcal{O}_E(1) \to 0
\] 
on $E$ such that $\bar P_1$ is a locally free sheaf of rank $2$ (\cite{NC}); 
$\bar P_1$ is defined by the following commutative diagram of exact sequences
\[
\begin{CD}
@. 0 @. 0 @. 0 @. \\
@. @VVV @VVV @VVV @. \\
0 @>>> \mathcal{O}_E(-1) @>>> \bar P_1(-L) @>>> \mathcal{O}_E(-1) @>>> 0 \\
@. @VVV @VVV @VVV @. \\
0 @>>> \mathcal{O}_E(-1) @>>> \mathcal{O}_E^{\oplus 2} @>>> \mathcal{O}_E(1) @>>> 0 \\
@. @. @VVV @VVV @. \\
@. @. \mathcal{O}_{C_2}(1) @>=>> \mathcal{O}_{C_2}(1) \\
@. @.  @VVV @VVV @. \\
@. @. 0 @. 0 @. 
\end{CD}
\]
Similarly there is a locally free sheaf $P_1$ of rank $2$ on $Y$ such that $P_1 \vert_E \cong \bar P_1$
defined by the following exact sequence
\[
0 \to P_1(-S) \to \mathcal{O}_Y^{\oplus 2} \to \mathcal{O}_S(C_0) \to 0
\]
where the right hand side arrow is obtained as a composition of surjective homomorphisms
$\mathcal{O}_Y^{\oplus 2} \to \mathcal{O}_S^{\oplus 2} \to \mathcal{O}_S(C_0)$.

We denote by $P_0 = \mathcal{O}_Y$ and $P_2 =  L$.
Let $s_j=\hat{\Omega}^j_E(j)[j]$ for $j = 0,1,2$, where
$\hat{\Omega}^1_E(1)$ denotes the double dual of $\hat{\Omega}^1_E \otimes \mathcal{O}_E(1)$.
Thus $s_0 = \mathcal{O}_E$ and $s_2 = \mathcal{O}_E(-2)[2]$.

\begin{Prop}\label{tilting P}
(1) The sum $P = \sum_{i=0}^2 P_i$ is a tilting generator of $D^b(\text{coh}(Y))$.

(2) $\{s_0,s_1,s_2\}$ is the set of simple objects in the category of perverse sheaves $\text{Perv}(Y/X)$ for 
$f: Y \to X$ defined by $P$ such that $\dim \text{Hom}(P_i,s_j) = \delta_{ij}$.
\end{Prop}

\begin{proof}
(1) We have $R^pf_*\mathcal{H}om(P_i,P_j) =0$ for $p > 0$ and for all $i,j$ except $i=j=1$, because
$H^p(E,\mathcal{O}_E(i)) = 0$ for $p > 0$ and $i \ge -3$ by the vanishing theorem (\cite{KMM}~Theorem~1.2.5) 
and $\mathcal{O}_E(-E) \cong \mathcal{O}_E(2)$.

We prove that $R^pf_*\mathcal{H}om(P_1,P_1) =0$ for $p > 0$.
There is an exact sequence
\begin{equation}\label{OQ}
0 \to \mathcal{O}_E \to \mathcal{H}om(\bar P_1,\mathcal{O}_E(1)) \to \mathcal{O}_E 
\to \mathcal{E}xt^1(\mathcal{O}_E(1),\mathcal{O}_E(1)) \to 0
\end{equation}
with $\mathcal{E}xt^1(\mathcal{O}_E(1),\mathcal{O}_E(1)) \cong \mathcal{O}_Q$ 
for the singular point $Q$ of $E$.
Since $H^p(E,\text{Ker}(\mathcal{O}_E \to \mathcal{O}_Q)) = 0$ for all $p$, 
we deduce that $H^p(E, \mathcal{H}om(\bar P_1,\mathcal{O}_E(1))) = 0$ for $p > 0$.
Therefore $H^p(E, \mathcal{H}om(\bar P_1,\bar P_1)) = 0$ for $p > 0$.
We can also check that, for $p > 0$ and $m > 0$, we have 
$H^p(E,\text{Ker}(\mathcal{O}_E(2m) \to \mathcal{O}_Q(2m))) = 0$, 
$H^p(E, \mathcal{H}om(\bar P_1,\mathcal{O}_E(1))(2m)) = 0$, and 
$H^p(E, \mathcal{H}om(\bar P_1,\bar P_1)(2m)) = 0$.
Therefore we have $R^pf_*\mathcal{H}om(P_1,P_1) =0$ for $p > 0$.

We prove that $P_0,P_1,P_2$ generate $D^b(\text{coh}(Y))$.
First, $\mathcal{O}_Y(iL)$ for $i=0,1,2$ generate $D^b(\text{coh}(Y))$ 
by \cite{BB2}~Lemma~4.2.4 or \cite{VdBergh}~Lemma~3.2.2.
By the exact sequence
\[
0 \to \mathcal{O}_Y(L) \to \mathcal{O}_Y(2L) \to \mathcal{O}_S(4C_0) \to 0
\]
we deduce that $\mathcal{O}_Y, \mathcal{O}_Y(L), \mathcal{O}_S(4C_0)$ generate $D^b(\text{coh}(Y))$.
Then by 
\[
\begin{split}
&0 \to \mathcal{O}_Y \to \mathcal{O}_Y(L) \to \mathcal{O}_S(2C_0) \to 0 \\
&0 \to \mathcal{O}_S(2C_0) \to \mathcal{O}_S(3C_0)^{\oplus 2} \to \mathcal{O}_S(4C_0) \to 0
\end{split}
\]
we deduce that $\mathcal{O}_Y, \mathcal{O}_Y(L), \mathcal{O}_S(3C_0)$ generate $D^b(\text{coh}(Y))$.
Finally, by
\[
0 \to P_1 \to \mathcal{O}_Y(L)^{\oplus 2} \to \mathcal{O}_S(3C_0) \to 0
\]
we conclude that $P_0,P_1,P_2$ generate $D^b(\text{coh}(Y))$.

\vskip 1pc

(2) We prove that $R\text{Hom}(P_i,s_j) \cong k^{\delta_{ij}}$, so that $s_j \in \text{Perv}(Y/X)$ and 
$\dim \text{Hom}(P_i,s_j) = \delta_{ij}$.
For $j = 0$, we have 
\[
\begin{split}
&R\text{Hom}(P_0,s_0) \cong R\Gamma(E,\mathcal{O}_E) \cong k \\
&R\text{Hom}(P_2,s_0) \cong R\Gamma(E,\mathcal{O}_E(-2)) \cong 0.
\end{split}
\]
Since $R\Gamma(E,\mathcal{O}_E(-1)) \cong 0$, we also have 
$R\text{Hom}(P_1,s_0) \cong R\Gamma(E,\bar P_1^*) \cong 0$.

For $j = 2$, we have 
\[
\begin{split}
&R\text{Hom}(P_0,s_2) \cong R\Gamma(E,\mathcal{O}_E(-2)[2]) \cong 0 \\
&R\text{Hom}(P_2,s_2) \cong R\Gamma(E,\mathcal{O}_E(-4)[2]) \cong k.
\end{split}
\]
Since $R\Gamma(E,\mathcal{O}_E(-3)) \cong 0$, we also have 
$R\text{Hom}(P_1,s_2) \cong 0$.

For $j = 1$, we have a distinguished triangle
\[
s_1[-1] \to \mathcal{O}_E^{\oplus 2} \oplus \mathcal{O}_E(-1) \to \mathcal{O}_E(1) \to s_1.
\]
Since $\text{Hom}(P_0,\mathcal{O}_E^{\oplus 2}) \cong \Gamma(E,\mathcal{O}_E^{\oplus 2}) \to 
\text{Hom}(P_0,\mathcal{O}_E(1)) \cong \Gamma(E,\mathcal{O}_E(1))$ is an isomorphism, 
we have $R\text{Hom}(P_0,s_1) \cong 0$.
We also have $R\text{Hom}(P_2,s_1) \cong 0$, because $R\Gamma(E,\mathcal{O}_E(-i)) \cong 0$ for $i = 1,2,3$.

We have $R\text{Hom}(P_1,\mathcal{O}_E) \cong R\text{Hom}(P_1,\mathcal{O}_E(-1)) \cong 0$.
Hence we have isomorphisms 
\[
R\text{Hom}_E(\bar P_1,\mathcal{O}_E(1)) \cong R\text{Hom}(P_1,\mathcal{O}_E(1))
\cong R\text{Hom}(P_1,s_1).
\]
$\bar P_1$ is a versal non-commutative deformation of $\mathcal{O}_E(1)$ on $E$ (\cite{NC}).
Therefore we have
$\text{Hom}_E(\bar P_1,\mathcal{O}_E(1)) \cong k$ and $\text{Hom}^1_E(\bar P_1,\mathcal{O}_E(1)) \cong 0$.
By duality, we have $\text{Hom}^2_E(\bar P_1,\mathcal{O}_E(1)) \cong \text{Hom}_E(\mathcal{O}_E(5),\bar P_1)^* 
\cong 0$.
Our claim is proved.
\end{proof}

Let $\mathcal{D}$ be the left orthogonal complement of an exceptional object $s_2 = \mathcal{O}_E(-2)[2]$ 
in $D^b(\text{coh}(Y))$:
\[
\mathcal{D} = \{a \in D^b(\text{coh}(Y)) \mid \text{Hom}(a, s_2[p]) = 0 \,\,\forall p\}.
\]
We can extend the concept of the tilting generators for triangulated categories such as $\mathcal{D}$,
and consider the categories of perverse coherent sheaves.

\begin{Prop}
$P_{\mathcal{D}} = P_0 \oplus P_1$ is a tilting generator of the triangulated category $\mathcal{D}$.
\end{Prop}

\begin{proof}
We prove that $P_0,P_1 \in \mathcal{D}$, and that $P_{\mathcal{D}}$ generate $\mathcal{D}$.
We already know that $\text{Hom}(P_{\mathcal{D}}, P_{\mathcal{D}}[p]) = 0$ for $p \ne 0$.

By the vanishing theorem, we have $R\Gamma(E,\mathcal{O}_E(i)) = 0$ for $i \ge -3$.
Thus $R\text{Hom}(P_0,\mathcal{O}_E(-2)) \cong R\Gamma(E,\mathcal{O}_E(-2)) \cong 0$.
By the duality, we have 
\[
R\text{Hom}(P_1,\mathcal{O}_E(-2)) \cong R\text{Hom}_E(\bar P_1,\mathcal{O}_E(-2)) \cong
R\text{Hom}_E(\mathcal{O}_E(2), \bar P_1[2])^* \cong 0.
\]
Hence $P_0,P_1 \in \mathcal{D}$.

There is an exact sequence
\[
0 \to P_2 \to P_0 \to \mathcal{O}_E \to 0.
\] 
Hence $P_0,P_1,\mathcal{O}_E$ generate $D^b(\text{coh}(Y))$.
Since $\omega_Y = \mathcal{O}_Y(-L)$, 
we have
$\text{Hom}(\mathcal{O}_E, a) \cong \text{Hom}(a,\mathcal{O}_E(-2)[3])^*$ for any $a \in D^b(\text{coh}(Y))$
by the Serre duality. 
Thus $a \in \mathcal{D}$ if and only if $R\text{Hom}(\mathcal{O}_E, a) = 0$.
Therefore $P_0,P_1$ generate $\mathcal{D}$.
\end{proof}

Let 
\[
\text{Perv}(\mathcal{D}) = \{a \in \mathcal{D} \mid \text{Hom}(P_{\mathcal{D}},a[p]) = 0 \text{ for } p \ne 0\}
\]
be the heart of a $t$-structure of $\mathcal{D}$ defined by the tilting generator $P_{\mathcal{D}}$.
We will determine the set of all simple objects in $\text{Perv}(\mathcal{D})$.

We have 
\[
\text{Hom}(s_0,s_2[1]) = \text{Hom}(\mathcal{O}_E, \mathcal{O}_E(-2)[3]) 
\cong \text{Hom}(\mathcal{O}_E,\mathcal{O}_E)^* \cong k
\]
by the Serre duality.
Let 
\[
0 \to s_2 \to s'_0 \to s_0 \to 0
\]
be the corresponding extension in $\text{Perv}(Y/X)$, the category of perverse coherent sheaves
defined by the tilting generator $P$ in $D^b(\text{coh}(Y))$.
Let $s'_1 = s_1$.

\begin{Prop}
$\{s'_0,s'_1\}$ is the set of all simple objects in $\text{Perv}(\mathcal{D})$, and 
$\text{Hom}(P_i,s'_j) \cong k^{\delta_{ij}}$ for $i,j = 0,1$.
\end{Prop}

\begin{proof}
We first prove that $s'_j \in \text{Perv}(\mathcal{D})$ for $j = 0,1$.
We have 
\[
\begin{split}
&\text{Hom}(s_2,s_2[p]) \cong \text{Hom}(\mathcal{O}_E,\mathcal{O}_E[p]) \cong 
\begin{cases} k &\text{ for } p = 0 \\ 
0 &\text{ for } p \ne 0. \end{cases} \\
&\text{Hom}(s_0,s_2[p]) \cong \text{Hom}(\mathcal{O}_E,\mathcal{O}_E[1-p])^* \cong 
\begin{cases} k &\text{ for } p = 1 \\ 
0 &\text{ for } p \ne 1. \end{cases}
\end{split}
\]
Therefore $\text{Hom}(s'_0,s_2[p]) \cong 0$ for all $p$, hence $s'_0 \in \text{Perv}(\mathcal{D})$.
We have
\[
\begin{split}
&\text{Hom}(\mathcal{O}_E(-1),s_2[p]) \cong \text{Hom}(s_2, \mathcal{O}_E(-1)[3-p])^* \cong 0 \\ 
&\text{Hom}(\mathcal{O}_E,s_2[p]) \cong \begin{cases} k &\text{ for } p = 1 \\ 
0 &\text{ for } p \ne 1. \end{cases} \\
&\text{Hom}(\mathcal{O}_E(1),s_2[p]) \cong \text{Hom}(\mathcal{O}_E,\mathcal{O}_E(1)[1-p])^* \cong 
\begin{cases} k^2 &\text{ for } p = 1 \\ 
0 &\text{ for } p \ne 1. \end{cases} 
\end{split}
\]
Moreover $\text{Hom}(\mathcal{O}_E(1),s_2[1]) \to \text{Hom}(\mathcal{O}_E^{\oplus},s_2[1])$ is an isomorphism.
Hence $R\text{Hom}(s_1,s_2) \cong 0$, and $s_1 \in \text{Perv}(\mathcal{D})$.

We prove that $\text{Hom}(P_i,s'_j) \cong k^{\delta_{ij}}$ for $i,j = 0,1$.
Then it follows that the $s'_j$ are simple.
By Proposition~\ref{tilting P}, we have $R\text{Hom}(P_0,s_2) \cong 0$.
Hence $R\text{Hom}(P_0,s'_0) \cong R\text{Hom}(P_0,s_0) \cong k$.
Since $R\text{Hom}(P_1,s_j) \cong 0$ for $j=0,2$, we have $R\text{Hom}(P_1,s'_0) \cong 0$.
We also have $R\text{Hom}(P_0,s_1) \cong 0$ and $R\text{Hom}(P_1,s_1) \cong k$. 
\end{proof}

\begin{Prop}
(1) $f_*P$ is not reflexive, but $f_*P_{\mathcal{D}}$ is reflexive.

(2) $f_*\mathcal{E}nd(P)$ and $f_*\mathcal{E}nd(P_{\mathcal{D}})$ are not Cohen-Macaulay.
\end{Prop}

\begin{proof}
(1) We have $f_*P_2 \subsetneq f_*P_2(E) \cong f_*P_0 \cong \mathcal{O}_X$, hence $f_*P_2$ is not reflexive.
On the other hand, $f_*P_1 \cong f_*P_1(E)$, hence $f_*P_1$ is reflexive.

(2) We have $R^2f_*\mathcal{E}nd(P)(K_Y) \ne 0$, because 
$\mathcal{H}om(P_2,P_0) \otimes \mathcal{O}_Y(K_Y) \cong \mathcal{O}_Y(-2L)$.
Hence $f_*\mathcal{E}nd(P)$ is not Cohen-Macaulay.

The second statement is more subtle.
We consider the exact sequence (\ref{OQ}) tensored with $\mathcal{O}_Y(-L)$.
Since $R\Gamma(E,\mathcal{O}_E(-2)) \cong 0$, 
we have $H^1(E,\text{Ker}(\mathcal{O}_E(-2) \to \mathcal{O}_Q(-2))) \ne 0$ and 
$H^2(E,\text{Ker}(\mathcal{O}_E(-2) \to \mathcal{O}_Q(-2))) = 0$. 
Then we deduce that $H^1(E, \mathcal{H}om(\bar P_1,\mathcal{O}_E(-1))) \ne 0$
and $H^2(E, \mathcal{H}om(\bar P_1,\mathcal{O}_E(-1))) = 0$.
Hence 
\[
H^1(E, \mathcal{H}om(\bar P_1,\bar P_1(-2))) \ne 0, H^2(E, \mathcal{H}om(\bar P_1,\bar P_1(-2))) = 0
\]
while $H^p(E, \mathcal{H}om(\bar P_1,\bar P_1(2m))) \cong 0$ for 
$p > 0$ and $m \ge 0$.
Therefore we have
\[
R^1f_*\mathcal{H}om(P_1,P_1)(K_Y) \ne 0
\]
hence $f_*\mathcal{E}nd(P_{\mathcal{D}})$ is not Cohen-Macaulay.
 \end{proof}

Therefore $f_*\mathcal{E}nd(P_{\mathcal{D}})$ is not homologically homogeneous as already proved 
in \cite{VdBergh} Example A.1. 

\section{Non-commutative deformations}

We recall the theory of multi-pointed non-commutative deformations of simple collections developed in \cite{NC}.

\begin{Defn}
The base ring is a direct product $k^r$ of the base field $k$ for a positive integer $r$.
We deform a set of objects $\{F_i\}_{i=1}^r$ in a $k$-linear abelian category.
It is said to be a {\em simple collection} if $\text{End}(F) \cong k^r$ with $F = \bigoplus_{i=1}^r F_i$, i.e., 
$\text{Hom}(F_i,F_j) \cong k^{\delta_{ij}}$.
\end{Defn}

The simple collections are defined in \cite{NC} as generalizations of simple sheaves.
If $r=1$ and $F$ is a sheaf, then a simple collection is nothing but a simple sheaf.  
Simple sheaves behave well under deformations.
For example, a stable sheaf is a simple sheaf.
A set of simple objects in a $k$-linear abelian category is automatically a simple collection.

\begin{Expl}(\cite{Bridgeland})
Let $f: Y \to X = \text{Spec}(R)$ be a projective birational morphism from a smooth $3$-fold 
whose exceptional locus $C$ is a 
smooth rational curve with normal bundle $N_{C/Y} \cong \mathcal{O}_C(-1)^{\oplus 2}$.
$\{\mathcal{O}_C, \mathcal{O}_C(-1)[1]\}$ is the set of simple objects above the singular point $x \in X$ 
in the category of perverse coherent sheaves ${}^{-1}\text{Perv}(Y/X)$.
For a point $y \in Y$ above $x$, there is an exact sequence 
\[
0 \to \mathcal{O}_C \to \mathcal{O}_y \to \mathcal{O}_C(-1)[1] \to 0
\]
in ${}^{-1}\text{Perv}(Y/X)$.
$\mathcal{O}_y$ becomes a stable object under a suitable Bridgeland stability condition determined by the 
values of the central charge on the set $\{\mathcal{O}_C, \mathcal{O}_C(-1)[1]\}$, 
and $Y$ is the corresponding moduli space.
If we take a different Bridgeland stability condition, then we obtain the flop of $Y$.
We refer to \cite{Toda} for related topics.
\end{Expl}

\begin{Defn}
The category of $r$-pointed Artin algebras $(\text{Art}_r)$ consists of $k^r$-algebras $R$ with augmentations:
\[
\begin{CD}
k^r @>e>> R @>p>> k^r
\end{CD}
\]
with $p \circ e = \text{Id}$, 
which are finite dimensional as $k$-vector spaces and such that the ideals $M = \text{Ker}(p: R \to k^r)$ 
are nilpotent. 

A (multi-pointed) {\em non-commutative (NC) deformation} of a simple collection 
$F = \bigoplus_{i=1}^r F_i$ over a parameter algebra $R \in (\text{Art}_r)$ 
is a pair $(F_R,\phi)$ consisting of an object $F_R$ with a left $R$-module structure and an
isomorphism $\phi: k^r \otimes_R F_R \to F$ such that $F_R$ is flat over $R$.

Let $(\hat{\text{Art}}_r)$ be the category of pro-objects $\hat R$ of $(\text{Art}_r)$;
$\hat R$ is a $k^r$-algebra with augmentation 
$k^r \to \hat R \to k^r$ such that $R_m = \hat R/M^{m+1} \in (\text{Art}_r)$ for all $m \ge 0$
and $\bigcap_{m > 0} M^m = 0$, 
where $M = \text{Ker}(p: R \to k^r)$.
We denote $\hat R = \varprojlim R_m$.

A {\em formal NC deformation} of $F$ over $\hat R$ is a pair $(\{F_{R_m}\}_{m \ge 0}, \{\phi_m\}_{m \ge 0})$ 
consisting of a series of NC deformations $F_{R_m}$ 
of $F$ over $R_m = \hat R/M^{m+1}$ with 
isomorphisms $\phi_{m+1}: R_m \otimes_{R_{m+1}} F_{R_{m+1}} \to F_{R_m}$ and $\phi_0: F_{R_0} \to F$.
Any NC deformation is considered to be a special case of a formal NC deformation whose parameter algebra is 
finite dimensional as a $k$-vector space.

A formal NC deformation $(\{F_{R_m}\}, \{\phi_m\})$ 
of $F$ is said to be {\em versal} if the following conditions are satisfied:

\begin{enumerate}
\item For any NC deformation $(F_R,\phi)$ of $F$, there are an integer $m$, a $k^r$-algebra homomorphism
$g: R_m \to R$ and an isomorphism $F_R \cong R \otimes_{R_m} F_{R_m}$
which is compatible with $\{\phi_m\}$ and $\phi$.

\item The induced homomorphism $R_m/M_{R_m}^2 \to R/M^2$ is uniquely determined by $(F_R,\phi)$.

\end{enumerate}
\end{Defn}

\begin{Defn}
An {\em iterated non-trivial extensions} of $F = \bigoplus_{i=1}^r F_i$ 
is a sequence of objects $\{G^n\}_{n = 0}^N$ with $G^n = \bigoplus_{i=1}^r G^n_i$
such that 

\begin{enumerate}
\item $G^0_i = F_i$ for all $i$.

\item For each $n < N$, there are $i_1 = i_1(n)$ and $i_2 = i_2(n)$ 
such that $G^{n+1}_i = G^n_i$ for $i \ne i_1$ and that 
there is a non-trivial extension
\[
0 \to F_{i_2} \to G^{n+1}_{i_1} \to G^n_{i_1} \to 0.  
\]
\end{enumerate}
\end{Defn}

\begin{Thm}
Let $\{G^n\}_{n = 0}^N$ be an iterated non-trivial extensions of a simple collection $F$, and let 
$R = \text{End}(G^N)$.
Then $G^N$ is an NC deformation of $F$ over $R$.
\end{Thm}

The point in the above theorem is that $\dim R = r+N$ as a $k$-vector space.
A versal deformation can be constructed by iterated universal extensions:

\begin{Thm}
Let $F = \bigoplus_{i=1}^r F_i$ be a simple collection.
Assume that $\dim \text{Ext}^1(F,F) < \infty$.
Define a sequence of objects $F^{(n)} = \bigoplus_{i=1}^r F^{(n)}_i$ by universal extensions
\[
0 \to \bigoplus_{j=1}^r \text{Ext}^1(F^{(n)}_i,F_j)^* \otimes F_j \to F^{(n+1)}_i \to F^{(n)}_i \to 0.
\]
Then the $F^{(n)}$ can be obtained by iterated non-trivial extensions of $F$, and 
the inverse limit $\varprojlim F^{(n)}$ is a versal NC deformation of $F$.
\end{Thm}

We note that the above exact sequences correspond to distinguished triangles
\[
F^{(n+1)}_i \to F^{(n)}_i \to \bigoplus_{j=1}^r \text{Hom}(F^{(n)}_i,F_j[1])^* \otimes F_j[1] \to F^{(n+1)}_i[1].
\]

\section{Versal deformations}

The reduced tilting generator $\bar P$ is recovered as a versal non-commutative deformation of the simple objects
in the category of perverse coherent sheaves $\text{Perv}(X/Y)$: 

\begin{Thm}\label{recover}
Let $f: Y \to X$ be a projective morphism between noetherian schemes,  
and let $P$ be a tilting generator for $f$.
Assume that $X = \text{Spec}(R)$ for a complete local ring $R$ whose residue field is isomorphic to the base field.
Let $\{P_i\}_{i=1}^m$ and $\{s_j\}_{j=1}^m$ be the sets of indecomposable projective objects and simple objects
in $\text{Perv}(Y/X)$.
Set $\bar P = \bigoplus_{i=1}^m P_i$ and $\bar A = f_*\mathcal{E}nd(\bar P)$.
Then $\bar P$ is the versal deformation of the simple collection $\bigoplus_{j=1}^m s_j$ with the
parameter algebra $\bar A$.
\end{Thm}

\begin{proof}
Since $X$ is the spectrum of a complete local ring, we have $\bar P = \varprojlim \bar P/\mathfrak m^n\bar P$ 
for the maximal ideal $\mathfrak m \subset R$.
Since $\bar A/\mathfrak m^n \bar A$ is finite dimensional as a vector space over the base field, 
we deduce that $\bar P/\mathfrak m^n\bar P$ is obtained by iterated (trivial or non-trivial) extensions of the $s_j$
in $\text{Perv}(Y/X)$ by Theorem \ref{BR}.
By construction, we have 
\[
\dim \text{Hom}(\bar P,s_j) = 1, \text{Hom}(\bar P,s_j[1]) = 0
\]
for all $j$.
We will prove our assertion using these two cohomological properties.

Let 
\begin{equation}\label{extensions}
\dots \to F^{m+1} \to F^m \to \dots \to F^1 \to F^0 = \bigoplus_{j=1}^n s_j
\end{equation}
be a sequence of 
surjective morphisms in $\text{Perv}(Y/X)$ corresponding to the iterated extensions toward $\bar P$
which passes through the quotients $\bar P/\mathfrak m^n\bar P$; we have exact sequences
\[
0 \to s_{j(m)} \to F^{m+1} \to F^m \to 0
\]
for each $m$, where $j(m)$ depends on $m$.

The first property $\dim \text{Hom}(\bar P,s_j) = 1$ implies that all the extensions are non-trivial.
Indeed if there is a trivial extension during the course, then there exists $j$ and $m$ such that 
$\dim \text{Hom}(F^m,s_j) \ge 2$.
Since $\bar P \to F^m$ is surjective, we deduce that $\dim \text{Hom}(\bar P,s_j) \ge 2$, a contradiction.    

By the combination with the second property $\text{Hom}(\bar P,s_j[1]) = 0$, 
we deduce that the formal deformation $\bar P$ is versal.
Indeed let $G \to F^m$ for some $m$ be any non-trivial extension by some $s_j$ 
corresponding to a non-trivial morphism $F^m \to s_j[1]$; we have 
\[
0 \to s_j \to G \to F^m \to 0.
\]
By an exact sequence
\[
\text{Hom}(\bar P, G) \to \text{Hom}(\bar P, F^m) \to \text{Hom}(\bar P, \bar s_j[1]) = 0
\]
we infer that the morphism $\bar P \to F^m$ can be lifted to a morphism $\bar P \to G$, so that $\bar P$ dominates this 
non-trivial extension.
Since a versal deformation can be obtained by a sequence of iterated non-trivial extensions, 
we conclude that $\bar P$ is a versal deformation.  
\end{proof}

In the case of Bridgeland and Van den Bergh, the non-commutative deformations in the null category 
$\mathcal{C}$ can be described by the following theorem:

\begin{Thm}\label{defBvdB}
Let $f: Y \to X$ be as in \S 3.

(A) Let $\{P_i\}_{i=0}^r$ and $\{s_j\}_{j=0}^r$ be the sets of indecomposable projective objects and simple objects
in ${}^{-1}\text{Perv}(Y/X)$ as in \S 3 (A).
Then the reduced tilting generator $P = \bigoplus_{i=0}^r P_i$ is relatively generated by global sections, i.e., the natural homomorphism $p: f^*f_*P \to P$ is surjective. 
Let $Q = \text{Ker}(p)$ be the kernel. 
Let $I$ be the two-sided ideal of $A = f_*\mathcal{E}nd(P)$ generated by endomorphisms of $P$
which can be factored in the form $P \to P_0 \to P$. 
Then the following hold:

\begin{enumerate}

\item $Q[1] \in {}^{-1}\text{Perv}(Y/X)$, 
and it is the versal deformation of the simple collection $\bigoplus_{j=1}^r s_j$.

\item The parameter algebra of the versal deformation $Q[1]$
is given by the following formula
\[
\text{End}(Q[1]) \cong A/I.
\]
\end{enumerate}

(B) Let $\{P'_i\}_{i=0}^r$ and $\{s'_j\}_{j=0}^r$ be the sets of indecomposable projective objects and simple objects
in ${}^0\text{Perv}(Y/X)$ as in \S 3 (B).
Let $P' = \bigoplus_{i=0}^r P'_i$ be the reduced tilting generator, 
let $p': f^*f_*P' \to P'$ be the natural homomorphism, and 
let $Q' = \text{Coker}(p')$ be the cokernel. 
Let $I'$ be the two-sided ideal of $A' = f_*\text{End}(P')$ generated by endomorphisms of $P'$
which can be factored in the form $P' \to P'_0 \to P'$. 
Then the following hold:

\begin{enumerate}

\item $Q' \in {}^0\text{Perv}(Y/X)$, 
and it is the versal deformation of the simple collection $\bigoplus_{j=1}^r s'_j$.

\item The parameter algebra of the versal deformation $Q'$
is given by the following formula
\[
\text{End}(Q') \cong A'/I'.
\]
\end{enumerate}
\end{Thm}

\begin{proof}
(A) In the exact sequence (\ref{def Pi}), 
$L_i$ is relatively generated by global sections and $R^1f_*\mathcal{O}_Y = 0$.
Hence $\tilde P_i$ is also relatively generated by global sections.

We have an exact sequence
\begin{equation}\label{defQ}
0 \to Q \to f^*f_*P \to P \to 0
\end{equation}
in $(\text{coh}(Y))$.
We have $f_*f^*f_*P \cong f_*P$ by the projection formula.
Hence $f_*Q = 0$.
Since $R^1f_*\mathcal{O}_Y = 0$ and the fiber dimension is $1$, we deduce that $R^1f_*f^*f_*P = 0$.
Hence $R^1f_*Q = 0$.
Thus $Rf_*Q = 0$ and $Q[1] \in {}^{-1}\text{Perv}(Y/X)$.

As in the proof of Theorem~\ref{recover}, $Q[1]/\mathfrak m^nQ[1]$ for any $n$ can be expressed by a series of
(trivial or non-trivial) iterated extensions of the $s_j$ for $0 \le j \le r$.
We claim that $s_0$ does not appear in this series.
This follows from the following facts: 
\[
\text{Hom}(P_0,Q[1]) = R^1f_*Q = 0, 
\text{Hom}(P_0,s_0) \cong k, \text{Hom}(P_0,s_j[1]) = 0 \,\,\forall j.
\]
Indeed if $s_0$ appears in the series of extensions $\{F^m\}$ as in (\ref{extensions}), 
then $\text{Hom}(P_0, F^{m_0}) \ne 0$ 
for some $m_0$, since $\text{Hom}(P_0,s_0) \cong k$.
Since $\text{Hom}(P_0,s_j[1]) = 0$, the natural homomorphisms 
\[
\text{Hom}(P_0, F^{m+1}) \to \text{Hom}(P_0, F^m)
\]
are surjective for all $m \ge m_0$.
Then $\text{Hom}(P_0,Q[1]) = \varprojlim \text{Hom}(P_0,Q[1]/\mathfrak m^n Q[1]) \ne 0$, a contradiction.

Now we prove that $Q[1]$ is the versal deformation of the simple collection $\bigoplus_{j=1}^r s_j$.
By the argument of the proof of Theorem~\ref{recover}, it is sufficient to prove the following claim:
\[
\text{Hom}(Q[1],s_j) \cong k, \text{Hom}(Q[1],s_j[1]) = 0 \text{ for } 1 \le j \le r.
\]
Since $\text{Hom}(P,s_j) \cong k$ and $\text{Hom}(P,s_j[q]) = 0$ for $q \ne 0$, our claim is reduced to 
the assertion that $\text{Hom}(f^*f_*P,s_j[t]) = 0$ for $t=-1,0$ and $1 \le j \le r$.
This is equivalent to saying that $\text{Hom}(f^*f_*P,\mathcal{O}_{C_j}(-1)[t]) = 0$ for $t=0,1$.

Let $G_1 \to G_0 \to f_*P \to 0$ be an exact sequence with free $\mathcal{O}_X$-modules $G_i$ for $i=0,1$.
Then we have exact sequences $f^*G_1 \to G' \to 0$ and $0 \to G' \to f^*G_0 \to f^*f_*P \to 0$ for some $G'$.
Since $Rf_*\mathcal{O}_{C_j}(-1) \cong 0$, we have 
$\text{Hom}(f^*G_i,\mathcal{O}_{C_j}(-1)[t]) = 0$ for all $i,j,t$.
Then $\text{Hom}(G',\mathcal{O}_{C_j}(-1)) = 0$ and we obtain our claim.
 
\vskip 1pc

We prove the second assertion that $\text{End}(Q[1]) \cong A/I$.
The exact sequence (\ref{defQ}) in the category $(\text{coh}(Y))$ of coherent sheaves 
defines a distinguished triangle 
\[
f^*f_*P \to P \to Q[1] \to f^*f_*P[1]
\]
in $D^b(\text{coh}(Y))$.
Since $R^1f^*f_*P \cong 0$ and $\text{Hom}(f^*f_*P,\mathcal{C}) \cong \text{Hom}(f_*P,f_*\mathcal{C}) \cong 0$, 
we have $f^*f_*P \in {}^{-1}\text{Perv}(Y/X)$ by (\ref{Perv-1}).
Therefore the above distinguished triangle yields an exact sequence 
\[
0 \to f^*f_*P \to P \to Q[1] \to 0
\]
in the abelian category ${}^{-1}\text{Perv}(Y/X)$.
In particular, we have a surjective morphism $h: P \to Q[1]$ in ${}^{-1}\text{Perv}(Y/X)$.
The induced surjective morphism $P \to Q[1] \oplus s_0$ is 
a natural morphism from a versal deformation of a simple collection $\bigoplus_{j=0}^r s_j$ to another 
non-commutative deformation.
Let $h_*: \text{End}(P) \to \text{End}(Q[1])$ be the corresponding surjective homomorphism of 
the parameter rings.

Let $I_1 = \text{Ker}(h_*)$.
We have to prove that $I_1 = I$.
If an endomorphism $g: P \to P$ factors through $P_0$, then the composition $h \circ g: P \to Q[1]$ 
vanishes, since $\text{Hom}(P_0,Q[1]) = 0$.
Therefore $h_*(g) = 0$, and $I \subset I_1$.

Conversely, assume that $g \in I_1$, i.e., $h_*(g) = 0$.
Then there is the following commutative diagram in the category $(\text{coh}(Y))$:
\[
\begin{CD}
0 @>>> Q @>>> f^*f_*P @>p>> P @>>> 0 \\
@. @V0VV @VVV @VgVV \\
0 @>>> Q @>>> f^*f_*P @>p>> P @>>> 0.
\end{CD}
\]
By the diagram chasing, we find a homomorphism $\tilde g: P \to f^*f_*P$ such that 
$g$ is factored as $g = p \circ \tilde g$.

Let $G_1 \to G_0 \to f_*P \to 0$ be an exact sequence with locally free $G_i$ as in the first part of the proof,
and let $G' = \text{Im}(f^*G_1 \to f^*G_0)$.
We claim that the homomorphism $\tilde g: P \to f^*f_*P$ can be lifted to $\tilde g_0: P \to f^*G_0$.
Indeed, since $\text{Hom}(P,f^*G_1[1]) = 0$, we obtain $\text{Hom}(P,G'[1]) = 0$, because the fiber dimension of
$f$ is $1$ and that $f^*G_1$ and $G'$ are sheaves.
Then the homomorphism $\text{Hom}(P,f^*G_0) \to \text{Hom}(P,f^*f_*P)$ is surjective.
Therefore $g$ is factored through a direct sum of $P_0$, hence $g \in I$.
Therefore $I = I_1$, and the theorem is proved.

\vskip 1pc

(B) We have exact sequences
\begin{equation}\label{defQ'}
\begin{split}
&0 \to H_1 \to P' \to Q' \to 0 \\
&0 \to H_2 \to f^*f_*P' \to H_1 \to 0
\end{split}
\end{equation}
in $(\text{coh}(Y))$ for some sheaves $H_1, H_2$.
We have $f_*f^*f_*P' \cong f_*P'$ by the projection formula.
Hence $f_*f^*f_*P' \cong f_*H_1 \cong f_*P'$.
Since $R^1f_*f^*f_*P' = 0$ and the fiber dimension is $1$, we deduce that $R^1f_*H_1 = 0$.
Hence $f_*Q' = 0$.
Since $R^1f_*P' = 0$, we have $R^1f_*Q' = 0$.
Therefore $Rf_*Q' = 0$ and $Q' \in {}^0\text{Perv}(Y/X)$.

As in the proof of case (A), $Q'$ can be expressed by a series of
iterated extensions of the $s'_j$ for $1 \le j \le r$, since
$\text{Hom}(P'_0,Q') = 0$.

Now we prove that $Q'$ is the versal deformation of the simple collection $\bigoplus_{j=1}^r s'_j$.
It is sufficient to prove the following claim:
$\text{Hom}(Q',s'_j) \cong k$ and $\text{Hom}(Q',s'_j[1]) = 0$ for $1 \le j \le r$.
We have $\text{Hom}(P',s'_j) \cong k$ and $\text{Hom}(P',s'_j[q]) = 0$ for $q \ne 0$.
Since $H_1$ is a quotient of a direct sum of $P'_0$, we have $\text{Hom}(H_1,s'_j) = 0$ for $1 \le j \le r$.
Therefore we have our claim and the versality of $Q'$.

\vskip 1pc

We prove the second assertion that $\text{End}(Q') \cong A'/I'$.
Since $R^1f_*H_1 = 0$, 
we have also $H_1 \in  {}^0\text{Perv}(Y/X)$ by (\ref{Perv0}).
Hence the first exact sequence of (\ref{defQ'}) is an exact sequence in ${}^0\text{Perv}(Y/X)$.
In particular, we have a surjective homomorphism $h': P' \to Q'$ in ${}^0\text{Perv}(Y/X)$, which is 
a homomorphism of non-commutative deformations.
Let $h'_*: \text{End}(P') \to \text{End}(Q')$ be the corresponding homomorphism of the parameter rings of the
deformations.

Let $I'_1 = \text{Ker}(h'_*)$.
We have to prove that $I'_1 = I'$.
If an endomorphism $g: P' \to P'$ factors through $P'_0$, then the composition $h' \circ g: P' \to Q'$ 
vanishes, since $\text{Hom}(P_0,Q') = 0$.
Therefore $h'_*(g) = 0$, and $I' \subset I'_1$.

Conversely, assume that $g \in I'_1$.
Then there is the following commutative diagram in the category $(\text{coh}(Y))$:
\[
\begin{CD}
0 @>>> H_1 @>>> P' @>>> Q' @>>> 0 \\
@. @VVV @VgVV @V0VV \\
0 @>>> H_1 @>{p_1}>> P' @>>> Q' @>>> 0.
\end{CD}
\]
By the diagram chasing, we find a homomorphism $g_1: P' \to H_1$ such that 
$g$ is factored as $g = p_1 \circ g_1$.
Since $f_*f^*f_*P' \cong f_*H_1$ and $R^1f_*f^*f_*P' = 0$, we have $Rf_*H_2 = 0$, hence
$H_2 \in {}^0\text{Perv}(Y/X)$.
Thus the second sequence in (\ref{defQ'}) is also exact in ${}^0\text{Perv}(Y/X)$.
Since $P'$ is a projective object, we have $\text{Hom}(P',H_2[1]) = 0$.
Hence $\text{Hom}(P',f^*f_*P') \to \text{Hom}(P',H_1)$ is surjective, and $g_1$ is lifted to 
$\tilde g: P' \to f^*f_*P'$.

Let $G_1 \to G_0 \to f_*P \to 0$ be an exact sequence with free sheaves $G_i$ as in part (A) of the proof,
and let $G' = \text{Im}(f^*G_1 \to f^*G_0)$.
Since $R^1f_*f^*G_1 = 0$, we have $R^1f_*G' = 0$, and $G' \in {}^0\text{Perv}(Y/X)$.
Then $\text{Hom}(P',G'[1]) = 0$, and $\text{Hom}(P',f^*G_0) \to \text{Hom}(P',f^*f_*P')$ is surjective, 
and $\tilde g$ is lifted to a morphism through a direct sum of $P'_0$.
Thus $I'_1 \subset I'$, and this completes the proof.
\end{proof}

\begin{Cor}\label{opposite}
(1) $Q \cong Q'$.

(2) $A/I \cong A'/I' \cong (A/I)^o$, where the last term is an opposite ring.
\end{Cor}

\begin{proof}
(1) We have $s_j \cong s'_j[1]$ for $1 \le j \le r$ in $D^b(\text{coh}(Y))$.
Though the non-commutative deformations of the $s_j$ and $s'_j[1]$ are considered in different abelian categories, 
their deformations are the same.
Indeed the extension group $\text{Ext}^1(a,b)$ for $a,b \in D^b(\text{coh}(Y))$ is independent of the 
abelian categories containing $a,b$.
The corresponding distinguished triangles in $D^b(\text{coh}(Y))$ determine the extensions.
Therefore we have $Q \cong Q'$.

(2) We have $A/I \cong A'/I'$ by (1).
On the other hand, we have an order-reversing bijection $A \to A'$ which sends an endomorphism
$g: P \to P$ to its transpose ${}^tg: P^* \to P^*$.
If $g$ is factored as $P \to \mathcal{O}_Y \to P$, then ${}^tg$ is factored as
$P^* \to \mathcal{O}_Y \to P^*$.
Therefore the isomorphism $A^o \to A'$ induces an isomorphism $(A/I)^o \to A'/I'$. 
\end{proof}

\begin{Expl}[\cite{Donovan-Wemyss1}~Example~1.3]
Let $X = \text{Spec}k[[u,v,x,y]]/(u^2+v^2y-x(x^2+y^3))$ and let $f:Y \to X$ be a small crepant resolution.
Then $A/I \cong k\langle \langle x,y \rangle \rangle/(xy+yx,x^2-y^3)$.
\end{Expl}

The following is a generalization of a result of Donovan and Wemyss (\cite{Donovan-Wemyss3}):

\begin{Cor}\label{isolated}
Assume in addition that $f$ is a birational morphism.
Then $f$ is an isomorphism outside the closed fiber if and only if 
the parameter algebra of the versal deformation $A/I$ of the simple collection $\bigoplus_{j=1}^r s_j$
is finite dimensional as a vector space over the base field.  
\end{Cor}

\begin{proof}
We prove that the cosupport $\text{Supp}(A'/I')$ of $I'$ coincides with the discriminant locus $D \subset X$ of $f$, 
the set of scheme theoretic points on $X$ over which $f$ is not an isomorphism.
Then it follows that $A'/I'$ is finite dimensional if and only if $D$ consists of an isolated point.

If $x \not\in D$, then $f$ is an isomorphism near $x$.
Then $f^*f_*P' \to P'$ is an isomorphism near $x$, and $Q' = 0$ near $x$.
Therefore $x \not\in \text{Supp}(A'/I')$.

Conversely, assume that $x \in D$. 
Since the fiber $f^{-1}(x)$ is positive dimensional and 
$P'$ has a negative degree along the fiber, 
it follows that $P'$ is not generated by relative global sections. 
Then $Q' \ne 0$ near $x$, and $x \in \text{Supp}(A'/I')$.
\end{proof}

\begin{Que}
Let $C \cong \mathbf{P}^1$ be a smooth rational curve embedded in a $3$-dimensional complex manifold $Y$.
If $C$ is contractible complex analytically by a proper bimeromorphic morphism $f: Y \to X$ 
which is an isomorphism on $Y \setminus C$, then the parameter algebra of the versal non-commutative deformation
of $C$ in $Y$ 
is finite dimensional by the corollary.
Conversely, one can ask the following question: 
if the parameter algebra of the versal non-commutative deformation of $C$ in $Y$ 
is finite dimensional, then is $C$ contractible by a proper bimeromorphic morphism?
\end{Que}


Graduate School of Mathematical Sciences, University of Tokyo,
Komaba, Meguro, Tokyo, 153-8914, Japan 

kawamata@ms.u-tokyo.ac.jp

\end{document}